\documentclass[a4paper,12pt]{amsart}

\usepackage{amsthm,amssymb,amsmath,amscd}
\usepackage{amsthm}
\usepackage{pgf, tikz}
\usepackage{graphicx,subcaption}
\usepackage{verbatim}
\theoremstyle{definition}
\newtheorem{definition}{Definition}%[section]

\theoremstyle{plain}
\newtheorem{theorem}[definition]{Theorem}
{Lemma}
{Theorem}
\newtheorem{lemma}[definition]{Lemma}
\newtheorem{cor}[definition]{Corollary}

\newtheorem{claim}[definition]{Claim}
{Conjecture}

\theoremstyle{remark}

\newcommand{\pp}{P^3_3}
\newcommand{\ppkl}{P^{k}_\ell}
\newcommand{\pcz}{P^{4}_2}
\newcommand{\ff}{f^4_2}
\newcommand{\fff}{f^3_3}
\newcommand{\F}{\mathcal{F}_2^4}
\newcommand{\G}{\mathcal{F}^3_3}

\newcommand{\hH}{\hat{H}}
\newcommand{\hV}{\hat{V}}
\newcommand{\hE}{\hat{E}}

\newcommand{\sg}{\operatorname{sg}}

\newcommand{\cG}{\G}

\newcommand{\cc}{C}

\begin{document}

\title[Paths in hypergraphs]
{Paths in hypergraphs: a rescaling phenomenon}

\author{Tomasz {\L}uczak}

\address{Adam Mickiewicz University,
Faculty of Mathematics and Computer Science
ul.~Umultowska 87,
61-614 Pozna\'n, Poland}

\email{\tt tomasz@amu.edu.pl}

\author{Joanna Polcyn}

\address{Adam Mickiewicz University,
Faculty of Mathematics and Computer Science
ul.~Umultowska 87,
61-614 Pozna\'n, Poland}

\email{\tt joaska@amu.edu.pl}

\thanks{The first author partially 
supported by NCN grant 2012/06/A/ST1/00261. }

\keywords {Paths, hypergraphs, transition phenomenon}

\subjclass[2010]{Primary: 05D05, secondary:  05C35, 05C38, 05C65. }

\date{March 22, 2017}

\begin{abstract}
Let $P^k_\ell$ denote the loose $k$-path of length $\ell$ and let define  $f^k_\ell(n,m)$ 
as the minimum value of $\Delta(H)$ 
over all $P^k_\ell$-free $k$-graphs $H$ with 
$n$ vertices and $m$ edges. In the paper we 
study the behavior of $f^4_2(n,m)$ and $f^3_3(n,m)$ and characterize the structure of 
extremal hypergraphs.
In particular, it is shown that  
when $m\sim n^2/8$ the value of 
each of these functions drops down from $\Theta(n^2)$ to $\Theta(n)$.  
\end{abstract}

\maketitle

\section{Introduction}
In extremal graph theory we often  study  functions which emerge
when we appropiately scale the extremal parameters of graphs. A typical example is the minimum number of triangles, scaled by 
$\binom n3$,  in graphs on $n$ vertices and  density $p$. 
The celebrated result of Razborov~\cite{Ra} gives a full description of this function as a function of $p$; 
in particular he showed  it is smooth everywhere except at the points $1-1/t$ for integer~$t$
(for a similar result on cliques of larger size see Reiher~\cite{R}).
  Thus, at these points a kind of the continuous phase transition takes place, which 
is related to the structural changes of  the graph on which the minimum is attained. 

It is  not too hard to construct examples which exhibits a much more rapid, discontinuous change of the structure. In this paper however we give examples of two functions 
where not only such a transition is discontinuous but the studied function  rapidly drops to zero and so requires another rescaling. 

In order to state our result we need  a few definitions. 
By a $k$-uniform hypergaph $H=(V,E)$ on $n$ vertices or, briefly, $k$-graph, we mean the family of $k$-element subsets (called edges) of a set of vertices of $H$. 
Let $\ppkl$ denote the loose $k$-uniform path of length $\ell$, i.e. the connected linear $k$-graph with $\ell$ edges and $k\ell-\ell+1$ vertices. 
Our aim is to exhibit a `rescaling phenomenon' for  the maximum degree in 4-graphs which contains no loose paths of length two and 3-graphs without loose paths of length three. In particular we prove the following two results (for more precise statements see Theorems~\ref{th1a}, \ref{th2a} below).

\begin{theorem}\label{th1}
There exists $n_1$ such that for every $P_2^4$-free $4$-graph $H$ with $n\ge n_1$ vertices and $m\ge \binom{\lfloor n/2\rfloor}2+1$ edges   we have $\Delta(H)\ge n^2/32-n$. 

On the other hand, for every $n\ge 4$  
 there exists a $P_2^4$-free 4-graph $H_0$ with $m= \binom{\lfloor n/2\rfloor}2$ edges 
and $\Delta(H)= \lfloor n/2\rfloor-1$. 
 \end{theorem}
      
\begin{theorem}\label{th2}
There exists $n_2$ such that for every $P_3^3$-free $3$-graph $H$ with $n\ge n_2$ vertices and 
$m\ge n^2/8 +1$ edges  we have $\Delta(H)\ge n^2/32-n$. 

On the other hand, for every $n\ge 4$  
 there exists a $P_3^3$-free 3-graph $H_0$ with $m=\lfloor n^2/8\rfloor$ edges 
  and $\Delta(H)\le \lceil n/2\rceil $. 
 \end{theorem}

\section{Paths of length two in 4-graphs}

In this section we study the maximum degree of hypergraphs which contains no paths of 
length two. For 2-graphs the problem is trivial since each  graph without paths of length two clearly consists of isolated edges. For 3-graphs the problem is also not very exciting. 
It is easy to see that every component of 3-graph without $P^3_2$ is either a {\em 2-star}, i.e. consists of edges which contain two given vertices, or is a subgraph of
the complete 3-graph on four vertices. Since the latter graph is denser, 3-graph without paths of length two on $n$ vertices contains at most $\lfloor (n+1)/4\rfloor+3\lfloor n/4\rfloor$ edges and 
this maximum number is achieved, for instance,  for the 3-graph which consists of disjoint cliques of size four and, perhaps, one isolated edge (in the case  $n\equiv 3$ (mod 4)). Hence the minimum maximum degree of any  $P^3_2$-free graph is three.

For 4-graphs the problem starts to be interesting. Indeed, let us recall that, at least for large $n$, the maximum number of edges in $P^4_2$-free graph on $n$ vertices is $\binom {n-2}2$ and it is achieved only for 2-stars in which there is a vertex which is contained in every edge of 4-graph; more precisely the following result was proved by Keevash, Mubayi, and Wilson~\cite{KMW}.

\begin{theorem}\label{KMW}
If $h(n)$ denote the maximum number of edges in a $\pcz$-free 4-graph on $n$ vertices, then 
$$h(n)=\begin{cases}
\binom n4 &\quad \textrm{for}\quad n=4,5,6,\\
15 &\quad \textrm{for}\quad n=7,\\
 17 &\quad \textrm{for}\quad n=8,\\
\binom {n-2}2 &\quad \textrm{for}\quad n\ge 9.\\  
\end{cases} $$      
\end{theorem}

      In order to state our result precisely we introduce some notation. 
For $n$ large enough and $m\le \binom {n-2}2$ let function $\ff(n,m)$ be defined as 
\begin{multline*}
\ff(n,m)=\min\{\Delta(H): H=(V,E)\textrm{\ is a $4$-graph such that \ }\\
|V|=n, |E|=m, \textrm{\ and\ }H\not\supset P^4_2\},
\end{multline*}
where here and below $\Delta (H)$ denotes the maximum degree of $H$.
 By $\F(n,m)$ we denote 
the `extremal' family of $P^4_2$-free $4$-graphs on $n$ vertices and $m$ edges such that $\Delta(H)=\ff(n,m)$. By $\tilde K^4_n$ we mean the {\em thick $n$-clique}, i.e. the graph 
on $n$ vertices (almost) partitioned into $\lfloor n/2\rfloor $ `dubletons' such that any pair of dubletons form an edge of $\tilde K^4_n$. 
%The {\em thick trail of lenth $n$} $T^4_n$ is a 4-graph with edges $e_1,\dots, e_n$ such that for every $i=1,\dots, n-1$ we have $|e_i\cap e_{i+1}|=2$, and $e_i\cap e_j=\emptyset$ whenever $j-i\ge 2$. 
The main theorem of these section can be stated as follows.

\begin{theorem}\label{th1a}
There exists $\bar n_1$ such that for every $n\ge \bar n_1$ and 
$$\binom{\lfloor n/2\rfloor}2-\frac n5\le m\le \binom{\lfloor n/2\rfloor}2$$ each graph from $\F(n,m)$ is a subgraph of a thick clique. 

Moreover, there exist $\tilde n_1$  such that for every $n\ge \tilde n_1$  and all 
$m\ge \binom{\lfloor n/2\rfloor}2+1$ each graph from $\F(n,m)$ has the maximum degree
at least $n^2/32-n$ and one can delete from it at most $470$ edges to obtain a  union of at most four 2-stars and some number of isolated vertices.
 \end{theorem}

Clearly, Theorem~\ref{th1} follows from Theorem~\ref{th1a}. Before we  
present its proof  let us mention few of its other consequences. 
It is easy to see that if we want to minimize the maximum degree in union of  $r$ stars for a given $n$, $m$ and $r$, we need to make the $r-1$ largest stars roughly as equal as possible. On the other hand subgraphs of a thick clique can be made almost regular, so for small $m$
the function $\ff(n,m)$ decreases linearly with $m$.
This observations lead directly to the following result.   
 
\begin{cor}\label{cor1a}
For every $x\in [0,1/4)\cup (1/4,1]$ the limit 
$$f(x)=\lim_{n\to\infty} \frac{\ff(n,x \binom{n-2}{2})}{\binom{n-2}2}$$
exists and 
$$f(x)=\begin{cases}0&\textrm{\ for\ }\quad 0\le x< 1/4,\\
(1+2x+\sqrt{12x-3})/24&\textrm{\ for\ }\quad 1/4<x< 1/3,\\
(1+3x+2\sqrt{6x-2})/18&\textrm{\ for\ }\quad 1/3<x< 1/2,\\
(x+\sqrt{2x-1})/2&\textrm{\ for\ }\quad 1/2<x\le 1.\end{cases}
  $$
  
Moreover, for every $m\le \binom{\lfloor n/2\rfloor}2$ we have 
$$\left\lfloor\frac{4m}{n-1}\right\rfloor\le \ff(n,m)\le \left\lceil\frac{4m}{n}\right\rceil\,.\qed$$
\renewcommand{\qed}{}  
\end{cor}
 
 We note also that once the function $\ff(n,m)$ drops from $\Theta(n^2)$ to $\Theta(n)$ it becomes `more stable', i.e. the following result holds. 
 
\begin{cor}\label{cor1b}
For large enough $n$ the following holds. 
\begin{enumerate}
\item[(i)] If $\binom{\lfloor n/2\rfloor}2+4\le m\le \binom {n-2}2$, then %$f(n,m)=\Theta(n^2)$ and
$\ff(n,m-4)<\ff(n,m)$.  
\item[(ii)] If $\binom{\lfloor n/2\rfloor}2-\frac n{10}\le m\le \binom{\lfloor n/2\rfloor}2$, then $\ff(n,m)=\lfloor n/2\rfloor -1$. \qed
\end{enumerate}
\end{cor} 
 
The main ingredient of our argument is the following decomposition lemma which is 
true for all $\pcz$-free 4-graphs no matter what are their density.
 
	\begin{lemma}\label{podzialp4}
		For any $\pcz$-free 4-graph $H$ there exists a partition of its set of vertices $V=R\cup S\cup T$, such that subhypergraphs of $H$ defined as $H_R=\{h\in H:h\cap R\neq \emptyset\}$, $H_S=H[S]$ and $H_T=H\setminus(H_R\cap H_S)=\{h\in H[V\setminus R]: h\cap T\neq\emptyset\}$ satisfy:
		\begin{enumerate}
			\item[(i)]  $|H_R|\le 10|R|$,
			\item[(ii)] $H_S$ is a subgraph of a thick clique, and so $|H_S|\le {\lfloor |S|/2\rfloor\choose 2} $,
			\item[(iii)]  $H_T$ is a family of disjoint 2-stars such that centers of these stars are in $S$ whereas all other vertices are in $T$. In particular,    $|H_T|\le \binom{|T|}2$.
		\end{enumerate}
	\end{lemma}

\begin{proof}
	Let $H$ be a $\pcz$-free 4-graph with the set of vertices $V$, $|V|=n$, and the set of edges $E$, $|E|=m$. 
	We start with defining the set of exceptional vertices $R\subset V$. We put into $R$ vertices of degree at most ten one by one, until only vertices of degree at least eleven remain. Then, clearly,
	
	\begin{equation}\label{hrp4}
	|H_R|\le 10|R|,
	\end{equation}
	
Let us consider the 4-graph $\hH=H[V\setminus R]=(\hV, \hE)$. 
 For a set $S\subseteq \hV$  by its {\em signature} $\sg(S)$ 
we mean the projection of the edges of $\hH$ into $S$, i.e. 
$$\sg(S)=\{S\cap e: e\in \hE\}\,.$$
Our argument is based on the number of facts  on signatures of $e\in\hE$. 

	\begin{claim}\label{si1}
		The signature of each edge $e$ of $\hH$ contains no singletons and at least one dubleton.
		Moreover, each vertex of $e$ is contained in at least one element of the signature.
	\end{claim}
	
\begin{proof} The first part of the statement follows from the fact that $\hH$ is $\pcz$-free.
Now take $e\in \hE$. Since the degree of each vertex 	$v\in E$ is at least eleven, so it must be contained in at least one set from $\sg(e)$. Finally, if $\sg(e)$ contains no dubletons, 
then each edge $e'$ intersecting $e$ must share with it precisely three elements. But 
then, for  $v'=e'\setminus e$, the vertex  $v'\in \hV$ has degree at most four, 
contradicting our assumption on $\hH$.
\end{proof}

	\begin{claim}\label{si2}
	If the signature of an edge $e$ from $\hH$ contains two dubletons they are disjoint.
		\end{claim}

\begin{proof}
Let $e_1=\{x_1,x_2,x_3,x_4\}\in \hH$ and let $\{x_1,x_2\},\{x_2,x_3\}\in \sg(e_1)$. Then, there exist in $H$ two other edges, $e_2=\{x_1,x_2,y_1,y_2\}$ and $e_3=\{x_2,x_3,y_2,y_3\}$, where $y_1,y_2,y_3\notin e_1$. Set $V_1=e_1\cup e_2\cup e_3$. 
		We  argue that at least one of vertices in the component of $\hH$ 
		containing $V_1$  has degree at most 10
		contradicting the definition of $\hH$.

				Let us first consider the case where $y_1\neq y_3$. 
Since $H$ is $P^4_2$-free, the signature of $V_1$ contains no singletons, but one can easily verify 
that it cannot contain dubletons either. Note also that  
there exists an edge $e'$ not contained in $V_1$ but intersecting it, since otherwise,
 because of the degree restriction,  $V_1$ would contain at least $11\cdot 7/4>19$ edges, 
 contradicting Theorem~\ref{KMW}. 
Furthermore, one can check that to avoid $\pcz$, 
any edge $e'$ not contained in $V_1$ can intersect $V_1$ on one of ten possible triples.
But this means that the vertex $v=e'\setminus V_1$ has degree at most ten, contradicting the choice of $\hH$.
				
		Now let us assume that $y_1=y_3$. Note that to avoid 
$\pcz$ any edge containing $x_4$ not contained in $V_1$  must be of type $\{v,y_i,x_2,x_4\}$.
Since the degree of $x_4$ is at least eleven and it belongs to at most 
$\binom 53=10$ edges contained in $V_1$  such an edge, say,  
$e_4=\{v,y_1,x_2,x_4\}$ exists. But now any edge which intersect set $V_1$   on two vertices 
and does not contain $v$ creates a copy of $\pcz$ and there are only five triples which
added to $v'\notin V_1\cup \{v\}$ create no copy of $\pcz$. Since $V_1\cup \{v\}$ cannot be  
a component of $\hH$ (by the degree restriction such a component would contain more than  $15$ edges contradicting Theorem~\ref{KMW}), the assertion follows.		
\end{proof}

	\begin{claim}\label{si3}
The signature of no edge of $\hH$ contains a triple and a dubleton which intersect on one vertex.  
	\end{claim}	
	
\begin{proof}
If the signature of $e_1=\{x_1,x_2,x_3,x_4\}\in \hH$ contains a triple $\{x_1,x_2,x_3\}$ 
and a dubleton $\{x_3,x_4\}$, then there exist in $\hH$ two edges, $e_2=\{x_1,x_2,x_3,y_1\}$ and $e_3=\{x_3,x_4,y_1,y_2\}$, with $y_1,y_2\notin e_1$. But then the signature of $e_3$ contains two dubletons sharing exactly one vertex contradicting Claim~\ref{si2}.
\end{proof}	
	
	\begin{claim}\label{si4}
		Signature of each edge of $\hH$ consists either of two disjoint dubletons or one dubleton and two triples intersecting on this dubleton. 
	\end{claim}
	
\begin{proof} It is  a straightforward consequence of Claims~\ref{si1}-\ref{si3}.
\end{proof}	

Now we are ready to show Lemma~\ref{podzialp4}. 
We call a pair of vertices $\{x,y\}\subset \hV$ \emph{a twin} if there is no edge $e\in \hE$ such that  $|\{x,y\}\cap e|=1$.
  In other words, each edge of $\hE$ either contains both vertices $x$ and $y$, or none of them. 
%By singletons we mean all one element sets which belongs to a signature of some edge of $\hE$.
	Now let the set $S\subset \hV$ be the union of all twins in $\hH$, $|S|=s$, and $T=\hV\setminus S$. Observe that due to Claim~\ref{si4}, 
	an edge $e\in\hE$ is contained in $S$ (and thus belongs to $H_S$) if and only if 
its signature consists of two disjoint dubletons. Consequently, $H_S$ is a subgraph of a thick clique, and so $|H_S|\le {s/2\choose 2}$.
Moreover, it is easy to see 
that if $e\in \hE$ contains two triples intersecting on a dubleton, then the dubleton must 
be contained in  $S$ while two other vertices of $e$, which can be seperated by some edge, 
lie outside $S$, i.e. they belong to $T$.	
\end{proof}

\begin{proof}[Proof of Theorem~\ref{th1a}]
%Now we apply Lemma~\ref{podzialp4} to study the structure of dense graphs.
Let $H\in \F (n,m)$, $m\ge n^2/8-2n/3$, and let a partition $V=R\cup S\cup T$ and subgraphs $H_R$, $H_S$ and $H_T$ of $H$ be defined as in Lemma \ref{podzialp4}. Set $|R|=r$, $|S|=s$ and $|T|=t$. By Lemma~\ref{podzialp4},
\begin{equation}\label{mp4}
|H|=|H_R| + |H_S| + |H_T|\le 10r +{\lfloor s/2\rfloor\choose 2} +\frac{t^2}2 \le 10r+\frac{s^2}8-\frac s4+\frac{t^2}2.
\end{equation}
We start with the following claim.
\begin{claim}\label{tp4}
	If $T\neq \emptyset$ then there exists in $H_T$ a 2-star with at least $2m^2/n^2+m/n$ edges.
\end{claim}
\begin{proof}
	Let us define a 2-graph $G_T=(T,E_T)$ on the set of vertices $T$ putting $E_T=\{h\cap T: h\in H_T\}$. Note that $\delta(\hH)\ge 11$ and so $|T|\ge 12$.
	 Then the average degree of the graph $G_T$ is bounded below by
	\begin{align*}
	\frac{2|H_T|}{|T|}&=\frac{2(m-|H_R|-|H_S|)}{n-r-s}\ge 2\frac{m-10r-s^2/8}{n-r-s}\\
	&=\frac{2m}{n}+\frac {(2m/n)(r+s)-20r-s^2/4}{n-r-s}\ge \frac{2m}{n},
	\end{align*}
	where the last inequality follows by the facts that $m\ge n^2/8-2n/3$ and $r+s \le  n-|T|\le n-12$. Since any graph with average degree $d$ contains a component of at least $(d+1)d/2$ edges, the assertion follows.
\end{proof}
As an immediate consequence of the above fact we get the following result.
\begin{claim}\label{cl41}
	If $n^2/9\le m\le {\lfloor n/2\rfloor\choose 2}$, then $T=\emptyset$.
\end{claim}
\begin{proof}
	Note that a thick clique  on $n$ vertices has ${\lfloor n/2\rfloor\choose 2}$ edges and the maximum degree $\Delta \le n/2$. As a consequence, if for $n^2/9\le m\le {\lfloor n/2\rfloor\choose 2}$, $H\in \F(n,m)$, then $\Delta(H)\le  n/2<0.02n^2<2m^2/n^2$. Hence, by Claim \ref{tp4}, $T=\emptyset$.
\end{proof}

\begin{claim}\label{cl42}
	If $H_T=\emptyset$ then $m\le {\lfloor n/2\rfloor\choose 2}$. Furthermore, if 
in addition $H_R\neq \emptyset $, then  $m \le {\lfloor n/2\rfloor\choose 2}-n/5$.
\end{claim}
\begin{proof}
Let $H\in \F(n,m)$ be such that $H_T=\emptyset$.
Then the vertex set of $H$ can be partitioned into sets $S$ and $R$, where $|S|=s$, $s$ is even,  and $|R|=r=n-s$. 
The number of edges $m$ in such graph is bounded from above by $s^2/8+10r$.
It is easy to see that if $r\ge 2$ and $n$ is large enough this number is smaller than 
${\lfloor n/2\rfloor\choose 2} -{n}/5$. Let us consider now the case when $R$ consists of 
just one vertex $v$; note that in this case $n$ is odd. Suppose that $v$ belongs to 
  an edge $e$. Then $e$ must 
separate one twin $\{w,w'\}$ in $S$. But then each edge $e'$ of $H_S$ which contains the twin 
$\{w,w'\}$ must intersect $e$ on at least one more vertex and consequently $w$, as well as $w'$,
can be contained in at most two edges of $H_S$. Since $\deg(v)\le 10$ so, very crudely, $\deg(w), \deg (w')\le 12$. But then  
\begin{align*}
m&\le \binom {(n-3)/2}2+10+26 =\binom{(n-1)/2} 2 -\frac{n-3}2+36\\
&\le \binom{(n-1)/2} 2-\frac {n}5\,.\qed
\end{align*}
\renewcommand{\qed}{}
\end{proof}
  
Note that Claim~\ref{cl42} immediately implies the first part of Theorem~\ref{th1a}. To consider the second part let us assume that $H\in \F(n,m)$,  where $m>{\lfloor n/2\rfloor\choose 2}$.
Then, by Claim~\ref{cl42},   $T\neq \emptyset$. Consequently, by Claim \ref{tp4}, there exists in $H_T$ a 2-star with at least 
\begin{equation}\label{bigstarp4}
\frac{2m^2}{n^2}+\frac mn> \frac{2m}{n}\cdot\frac{n-4}{8}+\frac mn =\frac m4> \frac{n^2}{32}-\frac n8
\end{equation}
edges implying that $\Delta (H) > n^2/32 -n$ and $H_T$ contains at most three largest 2-stars.

\begin{claim}\label{6starsp4}
	$H_T$ consists of at most seven disjoint 2-stars.
\end{claim}
\begin{proof}
	Assume for a contradiction, that $H_T$ consists of at least eight disjoint 2-stars. Denote them by $S_i$, $i\ge 1$, where $\deg_H(v_i)\ge \deg(v_j)$, for $i<j$,  and $\{v_i,v'_i\}$ stands for a center of a 2-star $S_i$. Note that by (\ref{bigstarp4}),
	$$
	\deg_H(v_7)+\deg_H(v_8)< \frac 27\cdot \left(m-\frac m4\right)=\frac{3}{14}m < \Delta(H)-3.
	$$
	But then we can modify $H$ by switching edges of $S_8$ so they form one 2-star with $S_7$, removing one edge from each of the 2-stars, $S_1$, $S_2$, $S_3$ and add three edges to $S_7$ (note that since $\delta(\hH)\ge 11$, both $S_7$ and $S_8$ has at least twelve vertices each). Clearly, the modified graph $H'$ is $\pcz$-free, has $n$ vertices, $m$ edges but the maximum degree of $H'$ is smaller than the maximum degree of $H$, contradicting the fact, that $H\in \F(n,m)$.
\end{proof}

\begin{claim}
	One can delete from $H$ at most $470$ edges and get a union of at most 7 disjoint stars and some isolated vertices.
\end{claim}

\begin{proof}
	First we observe that  $r+s <48$. Indeed, if this is not the case we can modify $H$ by removing $\bar{m}=|H_R|+|H_S|\le 10r+s^2/8-s/4$ edges of $H_R\cup H_S$, delete one edge from each of the three largest 2-stars of $H_T$ and on the remaining $r+s-14$ vertices disjoint from the centers of 2-stars  build a new 2-star (or three 2-stars if $r+s>m/4$) with $\bar{m}+3$ edges. The $\pcz$-free graph obtained in this way would have the same number of edges but  the maximum degree smaller than $H$, contradicting the fact that $H\in \F(n,m)$. 
	
Since $r+s \le 47$ we have  $|H_R\cup H_S|\le 10r+s^2/8-s/4\le 470$ and therefore, by Claim \ref{6starsp4}, one can delete from a graph $H$ at most $470$ edges of $H_S\cup H_R$ to get a graph which is an union of at most seven disjoint 2-stars and some isolated vertices. 
\end{proof}

To complete the proof of Theorem \ref{th1a} we need to reduce the number of 2-stars from seven to four. Let $S_1$ be a 2-star with the largest number of edges in $H_T$. By (\ref{bigstarp4}), it has $t_1>n^2/32-n/8$ edges and therefore $n_1\ge n/4+2$ vertices. Note that there are no place on four such stars in the graph of $n$ vertices. In fact, the fourth 2-star must be build on at most $n/4-1$ vertices and consequently have $t_4\le n^2/32-7n/8+6\le t_1-3n/4+6$ edges.

Now suppose that a graph $H_T$ has at least 5 disjoint 2-stars. By $n_4$ and $n_5$ we denote the number of vertices in the forth and the fifth largest 2-stars of $H_T$. Without loss of generality $n_4\ge n_5$. Then we have $n_5<n/5$, so one can modify $H$ by removing from each of $S_1$, $S_2$, $S_3$ one edge, choosing two vertices of the smallest degree in $S_5$, removing $m'<2n/5$ edges incident to them, and joining them to $S_4$ increasing the number of edges in $S_4$ to at most $t_1-3n/4+6+2n/5+3 <t_1-n/3$. The resulting graph $H'$ has the maximum degree smaller than $\Delta(H)$, which contradict the assumption that $H\in \F(n,m)$. Consequently, removing $470$ edges results in a graph which contains at most four 2-stars and the assertion follows.   
\end{proof}

 \section{Paths of length three}

In this section we study the maximum degree of dense $\pp$-free 3-graphs.
As we see soon, both the results and their proofs are surprsingly similar to  
that presented in the previous  section.

For 2-graphs the problem is again an easy exercise -- 
a graph whose components are cycles of length three (except, perhaps, one isolated edge if $n\equiv 2$ (mod 3)) is the largest $P^2_3$-free graph on $n$ vertices and has the maximum degree two. Here, we concentrate on the first non-trivial case when we study 
the maximum degree of $P^3_3$-free 3-graphs.

The maximum number of edges in a $P^3_3$-free 3-graph on $n$ vertices for all $n$ 
was found by Jackowska, Polcyn and Ruci\'nski in \cite{JPRt}. 

\begin{theorem}\label{ex1}
	Let $\hat h(n)$ denote the maximum number of edges in a $P^3_3$-free 3-graph on $n$ vertices.
Then
	$$\hat h(n)=\begin{cases}
	\binom n3 &\quad \textrm{for}\quad n=3,4,5,6,\\
	20 &\quad \textrm{for}\quad n=7,\\
	\binom {n-1}2 &\quad \textrm{for}\quad n\ge 8\,. 
	\end{cases} $$      
\end{theorem}  

Let  
\begin{multline*}
\fff(n,m)=\min\{\Delta(H): H=(V,E)\textrm{\ is a $3$-graph such that \ }\\
|V|=n, |E|=m, \textrm{\ and\ }H\not\supset P^3_3\},
\end{multline*}
and let $\G(n,m)$  denote 
the `extremal' family of $P^3_3$-free $3$-graphs on $n$ vertices and $m$ edges such that $\Delta(H)=\fff(n,m)$. 
Moreover,  let us call a 3-graph $H$ {\em quasi-bipartite} if one can partition its set of vertices into three sets: $X=\{x_1,x_2,\dots, x_s\}$, $Y=\{y_1,y_2,\dots, y_s\}$, and 
$Z=\{z_1,z_2,\dots, z_t\}$ in such a way that all the edges of $H$  are of type 
$\{x_i,y_i,z_j\}$ for some $i=1,2,\dots, s$, $j=1,2,\dots, t$.  Finally, by a star with center $v$ we denote a 3-graph in which each edge contains $v$. 
Then the following holds. 

\begin{theorem}\label{th2a}
	There exists $\bar n_2$ such that for every $n\ge \bar n_2$,  and 
	$$n^2/8-\frac n5\le m\le n^2/8 \,,$$
	each graph from $\G(n,m)$ is  quasi-bipartite.

Moreover, there exists $\tilde n_2$ such that for every $n\ge \tilde n_2$ and 
	$$n^2/8< m\le \binom {n-1}2\,,$$ 
each graph from $\G(n,m)$ has the maximum degree at least $n^2/32$ and 
we can delete from it at most $144$ edges and get a union of 
	at most four stars and some number of isolated vertices. 
\end{theorem}

Observe that Theorem~\ref{th2} follows directly from Theorems~\ref{th2a}.
Another immediate  consequence  of the above two statements is the following result (note that the function $f(x)$ below is the same as the one  defined in 
Corollary~\ref{cor1a}).

\begin{cor}\label{cor2a}
	For every $x\in [0,1/4)\cup (1/4,1]$ the limit 
	$$f(x)=\lim_{n\to\infty} \frac{\fff(n,x \binom{n-1}{2})}{\binom{n-1}2}$$
	exists and 
	$$f(x)=\begin{cases}0&\textrm{\ for\ }\quad 0\le x< 1/4,\\
	(1+2x+\sqrt{12x-3})/24&\textrm{\ for\ }\quad 1/4<x< 1/3,\\
	(1+3x+2\sqrt{6x-2})/18&\textrm{\ for\ }\quad 1/3<x< 1/2,\\
	(x+\sqrt{2x-1})/2&\textrm{\ for\ }\quad 1/2<x\le 1.\qed\end{cases}
	$$
\end{cor}

The proof of Theorems \ref{th2a} follows closely the way we proved Theorem~\ref{th1a}.
Thus, as before, we start with  the following decomposition lemma.

\begin{lemma}\label{podzial}
	For any $\pp$-free 3-graph $H$ there exists a partition of its set of vertices
$V=R\cup S\cup T$, such that subhypergraphs of $H$ defined as $H_R=\{h\in H:h\cap R\neq \emptyset\}$, $H_S=H[S]$ and $H_T=H\setminus(H_R\cap H_S)=\{h\in H[V\setminus R]: h\cap T\neq\emptyset\}$ satisfy:
	\begin{enumerate}
		\item[(i)]  $|H_R|\le 6|R|$,   
		\item[(ii)] $H_S$ is quasi-bipartite, and so $|H_S|\le |S|^2/8$,
		\item[(iii)]  $H_T$ is a family of disjoint stars such that centers of these stars are in $S$ whereas all other vertices are in $T$, and so $|H_T|\le \binom{|T|}2$.
	\end{enumerate}
\end{lemma}

\begin{proof}
Let $H=(V,E)$ be a $\pp$-free 3-graph with  $|V|=n$ and  $|E|=m$. 
	We start with defining the set of `exceptional' vertices $R\subseteq V$. By a triangle $C$ we 
mean linear 3-graph with six vertices and three edges. First we include in $R$ all the components 
of $H$ which contain $C$. Then, from the remaining graph we move to  $R$ 
	 vertices of degree at most six one by one, until we end up with a graph $\hH$ of minimum degree at least  seven.
Then we set  $H_R=\{h\in H:h\cap R\neq\emptyset\}$ and define a graph $\hH=(\hV,\hE)$
by putting $\hV=V\setminus R$, $\hE=E\setminus H_R$. 

In order to estimate the number of edges in $H_R$ we need the following simple fact from \cite{JPR}.
 \begin{claim}\label{spojny}
	If $H$ is a connected $\pp$-free 3-graph on $n$ vertices containing~$\cc$,  then
	$
	|E(H)| \le 4n
	$.	
\end{claim}
  
Thus, 
the required bound $6|R|$ for the number of edges in $H_R$ follows. 

The main tool in proving Lemma~\ref{podzial} is, again,  an analysis of possible  signatures of edges in a 3-graph $\hH$, where as before, the signature of $e\in \hE$ is defined as the projection  
of $\hE$ onto $e$.

\begin{claim}\label{sii1}
Every vertex of $e\in \hE$ is covered by at least one set of the signature of~$e$.	 
\end{claim}

\begin{proof} It follows from the fact that $\delta(\hH)\ge 7>1$.
\end{proof}

\begin{claim}\label{sii2}
	The signature of none of the edges of $\hH$ contains two singletons.
\end{claim}
\begin{proof}
	Assume that  an edge $e=\{x_1,x_2,x_3\}\in \hE$ contains two singletons, say $x_1$ and $x_2$. Since $\hH$ is $\{\pp,\cc\}$-free, two edges that intersects $e$ on $x_1$ and $x_2$ must share two points, say $y_1$ and $y_2$. 
	
	Set $X=\{x_1,x_2,x_3,y_1,y_2\}$. Since the degree of $x_3$ is at least seven it must belong to  an edge $e'$ which is not contained in $X$. If $|e'\cap X|=1$ it would lead to a $\pp$, if 
	$e'\cap X=\{x_3,y_i\}$ it would create $C$. Hence, $e'$ must consists of $v\notin X$ and one of the vertices $x_1, x_2$. Let us assume that $e'=\{v,x_1,x_3\}$. Now consider possible candidates for edges $e''$ 
	which contain $v$. If for such an edge $|e''\cap X|\le 1$ then it leads to $\pp$, whereas 
	if $e''=\{v,x_i,y_j\}$ for some $i=1,2,3$, $j=1,2$, it creates a triangle $C$. 
	Thus, the only candidates for $e''$ are triples $\{v,y_1,y_2\}$, and $\{v,x_i,x_j\}$ 
	for $1\le i<j\le 3$. But it means that the degree of $v$ is at most four, while $\delta(\hH)\ge 7$. A contradiction. 
\end{proof}

\begin{claim}\label{sii3}
	If the signature of an edge $e\in \hH$ contains two dubletons, then their intersection is a singleton of $e$.
\end{claim}

\begin{proof}
	Let $e=\{x,y,z\}\in \hH$ and let $f=\{x,y,v_x\}$ and $f'=\{y,z,v_z\}$,
	denote two edges containing two dubletons $\{x,y\},\{y,z\}\in \sg(e)$. 
Suppose that $y$ is not a singleton of $\sg(e)$. 	
	By Claim~\ref{sii2}
we may assume that $x$ is not a singleton either. We first  argue that then there exists $f''=\{x,y,v\}$  such that $v\neq v_z$.  If $v_x\neq v_z$, then we can take just $f''=f$,
 so let $v_x=v_z$. 
Note that $x$ is  not a singleton in $\sg(e)$ so each edges containing it must contain some other 
vertex of~$e$. Moreover, any edge $e'=\{w, x,z\}$ with $w\neq v_z$ is prohibited since it contains two singletons $x$ and $z$. Thus, the existence of $f''$ follows from the fact that 
$\deg_H(x)\ge 7>3$. Now consider possible candidates for edges $e'$ containing~$v$. If $e'\cap e=\emptyset$ 
then it leads to either $\pp$ or~$C$. If $|e'\cap e|=1$ then it creates either  singletons $x$ or $y$ in $e$,
or the second (next to $y$) singleton $v$ in $f''$. Thus, all edges containing $v$ are contained in $e\cup \{v\}$, contradicting the fact that $\deg(v)\ge 7$.  
\end{proof}

\begin{claim}\label{sii4}
The signature of no edge from $\hH$ contains three dubletons.
\end{claim}

\begin{proof} It follows from Claims \ref{sii2} and~\ref{sii3}.
\end{proof}

\begin{claim}\label{sii5}
The signature of  an edge from $\hH$ consists either of disjoint singleton and dubleton, or
of two dubletons intersecting on a singleton.
\end{claim}

\begin{proof} It is a direct consequence of  Claims \ref{sii1}-\ref{sii4}.
\end{proof}

Now we can describe the partition of $\hV$ into $S$ and $T$. 
We call a pair of vertices $\{x,y\}\subset \hV$ \emph{a twin} if it cannot be separated by an edge $e\in \hE$, i.e. for no such edge $|\{x,y\}\cap e|=1$.  By singletons we mean all one-element sets which belong to a signature of some edge of $E$.

Now let $S\subset \hV$ consists of all twins and singletons of $\hH$, $|S|=s$, and $T=\hV\setminus S$. It is easy to see that an edge of $\hH$ is contained in $S$ 
if and only if  it has signature which consists of disjoint dubleton and singleton. All other edges 
belong to $H_T$. Note that each edge of $H_T$ contains a singleton which belong to $S$.

 Finally, note that any quasi-bipartite 3-graph on $s$ vertices contains at most 
$$m\le \max\{s'(s-2s'):s'\le s\}\le s^2/8\,,$$
edges, so $|H_S|\le s^2/8$.
\end{proof} 

\begin{proof}[Proof of Theorem~\ref{th2a}]
Since the argument is almost identical to the one from the proof of 
Theorem~\ref{th1a} we skip some technical details. 
Let $H\in \cG(n,m)$, $m\ge n^2/8-n/5$, and let a partition $V=R\cup S\cup T$ and subgraphs $H_R$, $H_S$ and $H_T$ of $H$ be as defined in Lemma~\ref{podzial}. Set $|R|=r$, $|S|=s$ and $|T|=t$. By Lemma~\ref{podzial},
\begin{equation}\label{m}
|H|=|H_R| + |H_S| + |H_T|\le 4r+s^2/8+t^2/2.
\end{equation}
We start with the following claim.
\begin{claim}\label{t}
	If $T\neq \emptyset$ then there exists in $H_T$ a star with at least $2m^2/n^2+m/n$ edges.
\end{claim}
\begin{proof}
Indeed, then the  2-graph $G_T=(T,E_T)$ defined on the set of vertices $T$ by 
taking   $E_T=\{h\cap T: h\in H_T\}$ has the average degree bounded from below by
	\begin{align*}
	\frac{2|H_T|}{|T|}&=\frac{2(m-|H_R|-|H_S|)}{n-r-s}\ge \frac{2(m-6r-s^2/8)}{n-r-s}=\\
	&\frac{2m}{n}+\frac {(2m/n)(r+s)-12r-s^2/4}{n-r-s}\ge \frac{2m}{n},
	\end{align*}
and so contains  a component of at least $2m^2/n^2+im/n$ edges.
\end{proof}
Since there exists a $\pp$-free quasi-bipartite graph with $n$ vertices and $\lfloor n^2/8\rfloor$ edges, the above result immediately implies the following fact.
\begin{claim}\label{granica}
	If $n^2/9\le m\le n^2/8$, then $T=\emptyset$.\qed
\end{claim}

On the other hand, since $H_R$ is sparse, it turns out that when $H_T=\emptyset$ the 
number of edges $H$ is bounded from below by $\lfloor n^2/8\rfloor$ and this maximum is achieved
only when $H_R=\emptyset$.

\begin{claim}\label{cl422}
	If $H_T=\emptyset$ then $m\le  n^2/8$. Furthermore, if 
in addition $H_R\neq \emptyset $, then  $m \le n^2/8-n/5$.\qed
\end{claim}

Now the first part of Theorem~\ref{th2a} follows directly from Claims~\ref{granica} and~\ref{cl422}. 
In order to show the second part of the assertion we can repeat, almost verbatim, 
the argument used in the proof of Theorem~\ref{th1a}. 
Thus, from   Claims~\ref{t} and~\ref{cl422}, it follows  that if $m>  n^2/8$ then 
$H_T$ contains a star with more than  $n^2/32+n/8$ edges. Consequently,  
 $H$ contains at most three vertices with maximum degree.
Then we infer that $H_T$  consists of at most six disjoint stars since otherwise we
could decrease the maximum degree of three largest ones by merging the sixth and seventh 
into one and add to them three edges taken from the biggest stars. 

Since $H_T$ consists of only few stars the sets $S$ and $R$ must be quite small (simple calculations show that $r+s< 25$) since otherwise we could remove all $\bar m$ edges inside it, 
take three edges from the largest stars, and on the set of $r+s-6$ vertices, where we excluded the
centers of stars of $H_T$, build a star with $\bar m+3$ edges. 
Then, the $\pp$-free graph constructed in this way would have 
the same number of edges as $H$ but smaller maximum degree, contradicting the fact 
that $H\in \cG(n,m)$. Since $r+s<25$, we have $|H_R\cup H_S|\le 6r+s^2/8<144$, i.e. we can remove from $H$ at most 144 edges and get a forest of at most 6 stars and, perhaps, some isolated vertices.

Finally, to complete the proof, it is enough to show that in fact $H_T$ 
consists of at most  four stars.  Indeed, otherwise we could modify a graph accordingly (by decreasing by one three largest stars and incorporate 
these three edges to small stars by shuffling their vertices) so we could keep its number of edges and $\pp$-freeness but decrease by one  its  maximum degree.    
\end{proof}

\section{Final remarks and comments}

It is easy to see that the constant 470 in Theorem~\ref{th1a} is far from being optimal. The reader can easily rewrite the proof to replace it by, say, 30. 
However  finding the smallest possible value of this constant requires more work and 
studying quite a few cases of small 4-graphs. Since it is not crucial for the main result, we just give the examples of the extremal 4-graphs we  have found. 

Let $H_1^4$ be a 4-graph on $3k+8$ vertices, $k\ge 100$, which consists of three complete disjoint 2-stars on $k$ vertices each, three edges joining centers of these stars, and a copy of the unique $\pcz$-free graph $F^4_{1,3}$ on the vertex set $\{x_1,\dots,x_8\}$  with $17$ edges found in \cite{KMW} whose set of 4-edges consists of all 4-element subsets of $\{x_1,\dots, x_8\}$ 
which have at least three elements in $\{x_1,\dots,x_4\}$. Then, to make $H_1^4$ a union of disjoint 2-stars, we need to remove three edges joining the centers of three large 2-stars and at least eight edges from $F^4_{1,3}$. It seems that one can always delete at most eleven 4-edges from 
a dense enough $\pcz$-free graph to get a union of at most 4-stars (and, perhaps, some number of isolated vertices), so the graph $H_1^4$ defined above is in a way extremal. It is however not unique -- one can modify it removing from each 2-star the same number $i\le k/10$ of edges to get another extremal example.     

On the other hand, if we want to get a union of four stars instead of four 2-stars, it is 
enough to remove from  $H_1^4$ only seven edges. However, $H_1^4$  is not extremal for the variant of this problem. A 4-graph 
$H_2^4$  on $3k+4$ which consists of a thick clique on $10$ vertices and three equal complete 2-stars rooted on its vertices needs at least eight edges to be deleted to become a union of at most four stars. The same is true for a 4-graph $H_3^4$ on $3k+6$ vertices which consists of three complete stars on $k$ vertices each, three edges joining their centers, and the complete clique on six vertices.

In a similar way one can try to improve the constant $144$ in Theorem~\ref{th2a}.  
Since  the structure of $P^3_3$-free $3$-graphs is well studied, one can use Theorem~\ref{th2a}
to replace 144 by just 10, and the extremal graph consists of three equal stars and the clique on six vertices.

Another, much more interesting question, is whether a similar rescaling phenomenon can be observed for other extremal problems. There is a number  of candidates for such a behaviour, we just mention two possible directions which follow the line of research initiated by this work. The first one concerns linear 3-paths $P_\ell^3$ of length $\ell$, for $\ell\ge 3$. It is known~\cite{KMV} that the largest number of edges in a $P_\ell^3$-free graph on $n$ vertices is 
$(1/2+o(1))n^2$ and the extremal graph contains vertices of degree $\Omega(n^2)$. Thus, 
since a thick clique is $P_\ell^3$-free, one can expect that this maximum degree drops to $O(n)$ at $m\sim n^2/8$. 

It is also conceivable that one can  generalize of our 
result on  $P^4_2$-free 4-graphs in the following direction.
For $r\ge 1$ let $\F(4r;n,m)$ be a family of $(4r)$-graphs on $n$ vertices and $m$ edges in which no two edges share precisely $2r-1$ points. Frankl and F\"uredi~\cite{FF} proved that to maximize the number of edges in such a graph one needs to take the family of all sets which contain a given set on $2r$ vertices. Clearly, in such a graph the maximum degree is $m=\Theta(n^{2r})$. On the other hand a thick $(4r)$-clique on $n$ vertices, where we first partition vertex set into pairs and then choose $2r$ of them to form an edge, has 
$\binom{\lfloor n/2\rfloor}{2r}$ edges but its maximum degree is just $\binom{\lfloor n/2\rfloor-1}{2r-1}=\Theta(n^{2r-1})$. Thus, one expect  a rapid change of the (minimum) maximum degree at   $m=\binom{\lfloor n/2\rfloor}{2r}$.

\bibliographystyle{plain}

\end{document}